\documentclass[a4paper]{amsart}
\usepackage[all]{xy}
\usepackage{ifthen}
\usepackage{amssymb}
\usepackage{mathabx}

\SelectTips{eu}{}
\calclayout 

\title[Approximations and adjoints in homotopy
  categories]{Approximations and adjoints in homotopy categories}

\thanks{This research was partially supported by DFG Schwerpunkt SPP
  1388.}
\author[Henning Krause]{Henning Krause} 
\address{Fakult\"at f\"ur Mathematik\\ Universit\"at
  Bielefeld\\ 33501 Bielefeld\\ Germany.}
\email{hkrause@math.uni-bielefeld.de}

\theoremstyle{plain}
\newtheorem{lem}{Lemma}[section]
\newtheorem{prop}[lem]{Proposition}
\newtheorem{cor}[lem]{Corollary}
\newtheorem{conj}[lem]{Conjecture}
\newtheorem{thm}[lem]{Theorem}

\theoremstyle{remark}

\theoremstyle{definition}
\newtheorem{rem}[lem]{Remark}
\newtheorem{exm}[lem]{Example}
\newtheorem{defn}[lem]{Definition}
\newtheorem*{ackn}{Acknowledegement}

\numberwithin{equation}{section}

\DeclareMathOperator{\Abel}{Ab}

\DeclareMathOperator{\Qcoh}{Qcoh}
\DeclareMathOperator{\Coker}{Coker}

\DeclareMathOperator{\Hom}{Hom}
\DeclareMathOperator{\Ext}{Ext}
\DeclareMathOperator{\Ex}{Ex}
\DeclareMathOperator{\Lex}{Lex}
\DeclareMathOperator{\fp}{fp}
\DeclareMathOperator{\Fpinj}{Fpinj}
\DeclareMathOperator{\Mod}{Mod}

\DeclareMathOperator{\Ker}{Ker}

\DeclareMathOperator{\card}{card}

\DeclareMathOperator{\Inj}{Inj}
\DeclareMathOperator{\proj}{proj}
\DeclareMathOperator{\Proj}{Proj}

\DeclareMathOperator{\Mor}{Mor}
\DeclareMathOperator{\Flat}{Flat}

\renewcommand{\mod}{\operatorname{mod}}

\newcommand{\colim}[1]{\mathop{\mathrm{colim}}\limits_{#1}}

\newcommand{\lto}[1][{}]{\stackrel{#1}{\longrightarrow}} 
 
\newcommand{\xto}{\xrightarrow}

\newcommand{\Ab}{\mathrm{Ab}} 

\newcommand{\op}{\mathrm{op}}
 
\newcommand{\inc}{\mathrm{inc}}
\newcommand{\can}{\mathrm{can}}

\newcommand{\ac}{\mathrm{ac}}
\newcommand{\pac}{\mathrm{pac}}
\newcommand{\pure}{\mathrm{pur}}

\def\a{\alpha}
\def\b{\beta}

\def\k{\kappa}
\def\la{\lambda}

\def\A{{\mathcal A}}
\def\B{{\mathcal B}}
\def\C{{\mathcal C}}

\def\Oc{{\mathcal O}}

\def\Sc{{\mathcal S}}

\def\T{{\mathcal T}}

\def\bbX{\mathbb X}

\def\bbZ{\mathbb Z}

\def\bfC{\mathbf C}
\def\bfD{\mathbf D}

\def\bfK{\mathbf K}

\begin{document}

\begin{abstract}
We provide a criterion for the existence of right approximations in
cocomplete additive categories; it is a straightforward generalisation
of a result due to El Bashir. This criterion is used to construct
adjoint functors in homotopy categories. Applications include the
study of (pure) derived categories. For instance, it is shown that the
pure derived category of any module category is compactly generated.
\end{abstract}

\maketitle

\section{Introduction}

This note is motivated by recent work of Neeman and Murfet where
derived categories of flat modules are studied in various settings
\cite{Ne2008,Ne2006,Mu}.  One of the essential ingredients of their
work is the construction of approximations and adjoints for categories of
complexes.  It turns out that El Bashir's proof of the flat cover
conjecture \cite{BEE,El} leads to a systematic approach yielding
such approximations and adjoints. It is our aim in the present work
to explain this new approach and some of its consequences.

Let us mention a few applications in this introduction because they
are easily stated. Given any additive category $\A$, we denote by
$\bfK(\A)$ the category of cochain complexes in $\A$ with morphisms
the cochain maps up to homotopy. For a fixed Quillen exact structure
on $\A$, we denote by $\bfD(\A)$ the corresponding derived category.

The first result is an analogue of the flat cover conjecture for
complexes of quasi-coherent sheaves on a scheme; it has been
established in the affine case by Neeman \cite{Ne2006} and Enochs et
al.\ \cite{BEIJR}, and for noetherian schemes by Murfet \cite{Mu}.

\begin{thm}
\pushQED{\qed} Given any scheme $\bbX$, the inclusion $\bfK(\Flat
\bbX)\to\bfK(\Qcoh \bbX)$ admits a right adjoint.\qedhere
\end{thm}

This result is a consequence of the more general Theorem~\ref{th:adj}
which is formulated in the setting of locally presentable categories
in the sense of Gabriel and Ulmer \cite{GU}. Roughly speaking, any
cocomplete category with a sufficiently nice set of generators is
locally presentable. In particular, Grothendieck abelian categories
and module categories are locally presentable.

Most of the present work is done for locally presentable categories,
including the following result which establishes the `existence' of
the derived category of an exact category that is locally presentable.

\begin{thm}
\pushQED{\qed} Let $\A$ be an exact category that is locally
presentable, and suppose that exact sequences are closed under
filtered colimits.  Then the canonical functor $\bfK(\A)\to\bfD(\A)$
admits a fully faithful right adjoint. In particular, the category
$\bfD(\A)$ has small Hom-sets.\qedhere
\end{thm}

Let us consider as an example of particular interest for any ring $A$
the pure exact structure on its module category $\Mod A$; it is the
smallest exact structure on $\Mod A$ such that exact sequences are
closed under filtered colimits. This yields the pure derived category
$\bfD_\pure(\Mod A)$, studied for example by Christensen and Hovey
\cite{ChHo}. Note that it contains the usual derived category
$\bfD(\Mod A)$ as a full triangulated subcategory. The triangulated
category $\bfD(\Mod A)$ is well-known to be compactly generated, and
the inclusion $\proj A\to\Mod A$ of the category of all finitely
generated projective modules induces an equivalence $\bfK^b(\proj
A)\xto{\sim}\bfD(\Mod A)^c$ onto the full subcategory formed by all
compact objects. We have the following analogue for the pure derived
category.

\begin{thm}
\pushQED{\qed} Let $A$ be a ring.  The pure derived category
$\bfD_\pure(\Mod A)$ is a compactly generated triangulated category
and the inclusion $\mod A\to\Mod A$ of the category of all finitely
presented modules induces an equivalence $\bfK^b(\mod
A)\xto{\sim}\bfD_\pure(\Mod A)^c$ onto the full subcategory formed by
all compact objects. Moreover, the canonical functor $\bfD_\pure(\Mod
A)\to\bfD(\Mod A)$ admits left and right adjoints that are fully
faithful.\qedhere
\end{thm}

It seems appropriate to comment on the level of generality in this
work. Most of our results are stated for locally presentable
categories, even though the arguments work as well with little extra
effort for the more general class of accessible categories
\cite{MP,AdRo}. Also, no attempt has been made to formulate results in
terms of Quillen model structures. So we tried to keep the exposition
as elementary as possible, concentrating on basic ideas. The
interested and educated reader will have no problems to make the
appropriate generalisations.

This paper is organised as follows. Section~2 is devoted to studying
the existence of right approximations, generalising work of El
Bashir. These results are applied in Section~3 where right adjoints of
functors between homotopy categories are constructed. In particular,
derived categories of exact categories are studied.  The special case
of a pure derived category is discussed in Section~4. The final
Section~5 collects results on left approximations and left
adjoints. We end this note by stating a conjecture on fp-injective
modules which is an analogue of results of Neeman on flat modules.

\begin{ackn}
I would like to thank Jan \v{S}\v{t}ov\'i\v{c}ek and my student
Alexander Schmeding for stimulating discussions on the subject of this
work. In addition, I am grateful to Amnon Neeman for carefully reading and
commenting on a preliminary version of this paper.
\end{ackn}

\section{Right approximations}
Let $\A$ be an additive category and $\B$ a full additive subcategory.
In this section we present conditions such that every object $Y$ in
$\A$ admits a \emph{right $\B$-approximation}, that is, a morphism
$f\colon X\to Y$ with $X$ in $\B$ such that every morphism $X'\to Y$
with $X'$ in $\B$ factors through $f$.  Right approximations in
additive categories were introduced by Auslander and Smal{\o}
\cite{AS}, and independently by Enochs, using the term `precover'
\cite{E}.

The following theorem is our main result in this section; it is a
straightforward generalisation of a result due to El Bashir
\cite[Theorem~3.2]{El}.  In fact, one finds a plethora of criteria for
the existence of right approximations in the literature, generalising
the existence of flat covers in module categories \cite{BEE}. To the
best of our knowldege, all these criteria can be derived from the
following theorem.

\begin{thm}\label{th:approx}
Let $\A$ be an additive category that is locally presentable and let
$\B$ be a full additive subcategory that is closed under filtered
colimits. Suppose there exists a regular cardinal $\a$ such that $\B$
is closed under $\a$-pure subobjects or under $\a$-pure quotients.
Then each object in $\A$ admits a right $\B$-approximation.
\end{thm}

Recall that an additive category is \emph{locally presentable} if it
is cocomplete and admits a generator that is $\a$-presentable for some
regular cardinal $\a$ \cite{GU, AdRo}. For example, Grothendieck
abelian categories and module categories are locally
presentable. Further examples include categories of cochain complexes;
this will be relevant for our applications.

We need to make the following definition.\footnote{The subsequent
  definition of $\a$-pure mono/epimorphisms deviates from the standard
  one in terms of $\a$-presentable objects. The new definition seems
  to be more practical and coincides with the standard one for $\a\gg
  0$, provided the category is locally presentable.}

\begin{defn}
Let $\a$ be a regular cardinal. A morphism $X\to Y$ in an arbitrary
category is called
\begin{enumerate} 
\item \emph{$\a$-pure monomorphism}, if it is an $\a$-filtered colimit
of split monomorphisms,
\item \emph{$\a$-pure epimorphism}, if it is an $\a$-filtered colimit
of split epimorphisms, and
\item \emph{$\a$-terminal} if for every factorisation $X\to X'\to Y$
the morphism $X\to X'$ is invertible if it is an $\a$-pure
epimorphism.
\end{enumerate}
\end{defn}

Note that colimits of morphisms in a category $\A$ are taken in the
\emph{category of morphisms} $\Mor\A$. The objects of $\Mor\A$ are the
morphisms in $\A$ and the morphisms are the obvious commuting
squares. The term `filtered' without prefix is used to mean
`$\aleph_0$-filtered'.

We refer to \cite{AR2004} for basic facts about pure morphisms.  For
instance, suppose that $\A$ is a locally $\b$-presentable category and
let $\a\ge\b$ be a regular cardinal. Then a morphism $X\to Y$ is an
$\a$-pure epimorphism if and only if it induces a surjective map
$\Hom_\A(C,X)\to\Hom_\A(C,Y)$ for every $\a$-presentable object $C$ in
$\A$. Thus the usual notion of purity in a module category is obtained
by specialising $\a=\aleph_0$.
 
The crucial input for proving the theorem is the following result due
to El Bashir, which he establishes more generally for any Grothendieck
abelian category.

\begin{prop}[{\cite[Theorem~2.1]{El}}]\label{pr:ElBashir}
Let $\A$ be a module category (over a ring with several
objects). Given an object $Y$ in $\A$ and a regular cardinal $\a$,
the isomorphism classes of $\a$-terminal morphisms $X\to Y$ in $\A$
form a set.
\end{prop}

\begin{proof}
Let $\C$ be an additive category and $\A=\Mod\C$ the category of
\emph{$\C$-modules}, that is, additive functors $\C^\op\to\Ab$ into
the category of abelian groups.  Given a $\C$-module $X$, we define
its \emph{cardinality} to be $|X|=\sum_{C\in\C_0}\card X(C)$, where
$\C_0$ denotes a representative set of objects in $\C$.  

It follows from \cite[Theorem~5]{BEE} that for each cardinal $\la$,
there exists a cardinal $\k$ such that for any morphism $f\colon X\to
Y$ in $\A$ with $|X|\geq \k$ and $|Y|\leq\la$, there exists an
$\a$-pure submodule $0\neq U\subseteq X$ with $f|_U=0$. Let $X'=X/U$.
Then $f$ admits a factorisation $X\xto{u} X'\xto{v} Y$ with
$u$ an $\a$-pure epimorphism that is not invertible.  Thus any
$\a$-terminal morphism $X\to Y$ with $|Y|\leq\la$ satisfies $|X|<\k$.
\end{proof}

\begin{cor}
Let $\A$ be an additive category that is locally
$\b$-presentable for some regular cardinal $\b$.  Given an object $Y$
in $\A$ and a regular cardinal $\a\geq \b$, the isomorphism classes of
$\a$-terminal morphisms $X\to Y$ in $\A$ form a set.
\end{cor}
\begin{proof}
Let $\C$ be the full subcategory formed by all $\b$-presentable
objects in $\A$ and denote by $\Mod\C$ the category of
$\C$-modules. The functor $F\colon \A\to \Mod\C$ taking an object $X$
to $\Hom_\A(-,X)|_\C$ is fully faithful and preserves $\a$-filtered
colimits for every regular cardinal $\a\ge\b$ \cite[\S7]{GU}. Thus $F$
preserves $\a$-pure epimorphisms. Note that the image of $F$ is closed
under $\a$-pure quotients \cite[Proposition~13]{AR2004}.  It follows
that $F$ preserves $\a$-terminal morphisms. Now apply
Proposition~\ref{pr:ElBashir}.
\end{proof}

\begin{proof}[Proof of Theorem~\ref{th:approx}]
We follow El Bashir \cite{El}. Suppose that the category $\A$ is
locally $\b$-presentable. We may assume that $\a\ge\b$. Fix an object
$Y$ in $\A$ and a representative set of $\a$-terminal morphisms
$X_i\to Y$ ($i\in I$) with $X_i$ in $\B$. We claim that the induced
morphism $\coprod_{i\in I}X_i\to Y$ is a right $\B$-approximation. To
see this, choose a morphism $f\colon X\to Y$ with $X$ in
$\B$. Consider the pairs $(u,v)$ of morphisms $X\xto{u} X'\xto{v} Y$
with $X'$ in $\B$ such that $f=vu$ and $u$ is an epimorphism.  These
pairs are partially ordered if one defines $(u_1,v_1)\leq (u_2,v_2)$
provided that $u_2$ factors through $u_1$. An upper bound of a chain
of pairs $(u_i,v_i)$ is obtained by taking its colimit $(\bar u,\bar
v)$ with $\bar u=\colim{i} u_i$ and $\bar v=\colim{i} v_i$. Note that
any colimit of epimorphisms is again an epimorphism. In addition, one
uses that $\B$ is closed under filtered colimits. In a locally
presentable category, the epimorphisms starting in a fixed object
form, up to isomorphism, a set \cite[Satz~7.14]{GU}. Thus we can
choose a maximal pair $(\bar u,\bar v)$, using Zorn's lemma. The
maximality implies that $\bar v$ is an $\a$-terminal morphism, since
$\B$ is closed under $\a$-pure quotients. Thus $f$ factors through an
$\a$-terminal morphism $X_i\to Y$ and therefore through $\coprod_{i\in
I}X_i\to Y$. It remains to observe that $\B$ is closed under $\a$-pure
quotients if it is closed under $\a$-pure subobjects; this follows from
the subsequent Proposition~\ref{pr:acc}.
\end{proof}

The following proposition collects some facts which help to apply
Theorem~\ref{th:approx}.

\begin{prop}\label{pr:acc}
Let $\A$ be a locally presentable category.  For a full subcategory
$\B$ that is closed under filtered colimits, the following conditions
are equivalent:
\begin{enumerate}
\item There exists a set $\Sc$ of objects in $\B$ such that every
object of $\B$ is a filtered colimit of objects in $\Sc$.
\item The category $\B$ is accessible  \cite{MP,AdRo}.
\item There exists a regular cardinal $\a$ such that $\B$ is closed
under $\a$-pure subobjects.
\end{enumerate}
Moreover, these conditions imply that there exists a regular cardinal
$\a$ such that $\B$ is closed under $\a$-pure quotients.
\end{prop}

\begin{proof}
(1) $\Rightarrow$ (2): Fix a regular cardinal $\a$ such that $\A$ is
locally $\a$-presentable and each object in $\Sc$ is $\a$-presentable.
Denote by $\B_0$ the smallest full subcategory that contains $\Sc$ and
is closed under filtered colimits over diagrams having
cardinality less than $\a$. Observe that the objects in $\B_0$ are
$\a$-presentable. It is not difficult to check that each filtered
colimit of objects in $\B_0$ can be rewritten as an $\a$-filtered
colimit of objects in $\B_0$. Thus every object in $\B$ is an
$\a$-filtered colimit of objects in $\B_0$.  This means that $\B$ is
$\a$-accessible.

(2) $\Rightarrow$ (1): Suppose that $\B$ is $\a$-accessible for some
    regular cardinal $\a$. Let $\Sc$ be a representative set of
    $\a$-presentable objects. It follows that each object in $\B$ is
    an $\a$-filtered colimit of objects in $\Sc$. In particular, each
    object in $\B$ is a filtered colimit of objects in $\Sc$.

(2) $\Leftrightarrow$ (3): See \cite[Corollary~2.36]{AdRo}.

The last statement follows from \cite[Proposition~13]{AR2004}.
\end{proof}

In \cite{EE}, Enochs and Estrada have shown that quasi-coherent
sheaves admit flat covers (by viewing sheaves as representations of
appropriate quivers). This is a simple consequence of
Theorem~\ref{th:approx}.
 
\begin{exm}
Let $(\bbX,\Oc_\bbX)$ be a scheme and denote by $\Qcoh \bbX$ the
category of quasi-coherent $\Oc_\bbX$-modules. This is a Grothendieck
abelian category and therefore locally presentable.  Recall that a
quasi-coherent $\Oc_\bbX$-module $M$ is \emph{flat} if the functor
$M\otimes_{\Oc_\bbX}-$ is exact, and let $\Flat\bbX$ denote the full
subcategory consisting of all flat modules in $\Qcoh \bbX$. It is
easily checked that $\Flat\bbX$ is closed under filtered colimits and
$\aleph_0$-pure subobjects. Thus every quasi-coherent
$\Oc_\bbX$-module admits a right $\Flat\bbX$-approximation. A standard
argument \cite[\S7]{E} shows that one can choose the right
approximation to be minimal.
\end{exm}

\section{Right adjoints}

In this section we construct right adjoints of functors between homotopy
categories, applying the criterion for the existence of right
approximations from the previous section.

Let $\A$ be an additive category.  We denote by $\bfC(\A)$ the
category of \emph{cochain complexes}, that is, sequences of morphisms
$(d^n\colon X^{n}\to X^{n+1})_{n\in\bbZ}$ in $\A$ such that $d^{n}
d^{n-1}=0$ for all $n\in\bbZ$. The morphisms in this category are the
usual cochain maps. The \emph{homotopy category} $\bfK(\A)$ is the
category of cochain complexes with morphisms the cochain maps up to
homotopy.

\begin{lem}
Let $\A$ be a locally presentable additive category. Then the category
$\bfC(\A)$ is locally presentable. Moreover, all limits and colimits
in $\bfC(\A)$ are computed degreewise.
\end{lem}
\begin{proof}
Suppose that $\A$ is locally $\a$-presentable for some regular
cardinal $\a$. We denote by $\A^\bbZ$ the category consisting of all
sequences of morphisms $(d^n\colon X^{n}\to X^{n+1})_{n\in\bbZ}$ in
$\A$. Thus $\A^\bbZ$ equals the category of functors $\bbZ \to \A$,
where $\bbZ$ is viewed as category with exactly one morphism $i\to j$
if and only if $i\le j$. It follows that $\A^\bbZ$ is locally
$\a$-presentable \cite[Corollary~1.54]{AdRo}.  Note that all
(co)limits in $\A^\bbZ$ are computed degreewise and that the
subcategory $\bfC(\A)$ is closed under (co)limits. Moreover, $\bfC(\A)$
is closed under subobjects.  Any $\a$-pure monomorphism in $\A^\bbZ$
is a monomorphism, since $\A^\bbZ$ is locally $\a$-presentable.  Thus
$\bfC(\A)$ is closed under $\a$-pure subobjects, and it follows from
Proposition~\ref{pr:acc} that $\bfC(\A)$ is accessible. In fact,
$\bfC(\A)$ is complete and therefore locally presentable
\cite[Corollary~2.47]{AdRo}.
\end{proof}

Our tool for constructing right adjoints is the following proposition which
is due to Neeman.

\begin{prop}[{\cite[Proposition~1.4]{Ne2006}}]\label{pr:N}
\pushQED{\qed} Let $\T$ be a triangulated category and $\Sc$ a full
triangulated subcategory. Suppose that $\Sc$ and $\T$ have split
idempotents. Then the following are equivalent:
\begin{enumerate}
\item The inclusion $\Sc\to\T$ admits a right adjoint.
\item Every object in $\T$ admits a right $\Sc$-approximation.\qedhere
\end{enumerate}
\end{prop}

Note that any triangulated category has split idempotents provided the
category admits countable coproducts \cite[Proposition~1.6.8]{Ne}.

The next theorem is the analogue of Theorem~\ref{th:approx} for
homotopy categories.

\begin{thm}\label{th:adj}
Let $\A$ be a locally presentable additive category and $\B$ a full
additive subcategory. Suppose that $\B$ is closed under filtered
colimits and in addition closed under $\a$-pure subobjects or under
$\a$-pure quotients for some regular cardinal $\a$.  Then the
inclusion $\bfK(\B)\to\bfK(\A)$ admits a right adjoint.
\end{thm}
\begin{proof}
We view $\bfC(\B)$ as a full subcategory of $\bfC(\A)$. Colimits in
$\bfC(\A)$ are computed degreewise and this implies that $\bfC(\B)$ is
closed under filtered colimits and $\a$-pure subobjects or $\a$-pure
quotients, respectively.  Thus every object in $\bfC(\A)$ admits a
right $\bfC(\B)$-approximation by Theorem~\ref{th:approx}, and it
follows that every object in $\bfK(\A)$ admits a right
$\bfK(\B)$-approximation.  Thus we can apply Proposition~\ref{pr:N}
and conclude that the inclusion $\bfK(\B)\to\bfK(\A)$ admits a right
adjoint.
\end{proof}

The following result is an application; it has been established in the
affine case by Neeman \cite{Ne2006} and Enochs et al.\ \cite{BEIJR},
and for noetherian schemes by Murfet \cite{Mu}. Their proofs are
different from the one given here.

\begin{cor}
\pushQED{\qed} Given any scheme $\bbX$, the inclusion $\bfK(\Flat
\bbX)\to\bfK(\Qcoh \bbX)$ admits a right adjoint.\qedhere
\end{cor}

Let $\A$ be an \emph{exact category} \cite{Q}.  Thus $\A$ is an
additive category, together with a distinguished class of sequences
$X\xto{u} Y\xto{v} Z$ of morphisms which are called \emph{exact} and
satisfy a number of axioms. Note that the morphisms $u$ and $v$ in
each exact sequence as above form a \emph{kernel-cokernel pair}, that
is, $u$ is a kernel of $v$ and $v$ is a cokernel of $u$.  A morphism
in $\A$ which arises as the kernel in some exact sequence is called
\emph{admissible monomorphism}; a morphism arising as a cokernel is
called \emph{admissible epimorphism}. 

A full subcategory $\B$ of $\A$ is \emph{extension closed} if every
exact sequence in $\A$ belongs to $\B$ provided its endterms belongs
to $\B$.  Any full and extension closed subcategory of $\A$ is exact
with respect to the class of sequences which are exact in $\A$.

An object $P$ in $\A$ is \emph{projective} if each admissible
epimorphism $Y\to Z$ induces a surjective map
$\Hom_\A(P,Y)\to\Hom_\A(P,Z)$, and the full subcategory of $\A$ formed
by these objects is denote by $\Proj \A$. Analogously, the subcategory
$\Inj\A$ of injective objects is defined.

A cochain complex $X=(X^n,d^n)$ in $\A$ is called \emph{acyclic} if
for each $n\in\bbZ$ there is an exact sequence $Z^n\xto{u^n}
X^{n}\xto{v^n} Z^{n+1}$ in $\A$ such that $d^n=u^{n+1}v^n$. The full
subcategory consisting of all acyclic complexes in $\bfC(\A)$ is
denoted by $\bfC_\ac(\A)$. The acyclic complexes form a full
triangulated subcategory of $\bfK(\A)$ which we denote by
$\bfK_\ac(\A)$. Following \cite{Ne2000,Ke}, the \emph{derived
category} of $\A$ is by definition the Verdier quotient
$$\bfD(\A)=\bfK(\A)/\bfK_\ac(\A).$$

It is a well-known fact that the derived category of any Grothendieck
abelian category has small Hom-sets \cite{Gal,Be,Fr}. On the other hand,
there are simple examples of abelian categories where this property
fails \cite{CN}.  The following result establishes the `existence' of
the derived category of an exact category that is locally presentable.

\begin{thm}\label{th:derived}
Let $\A$ be an exact category that is locally presentable, and suppose
that exact sequences are closed under filtered colimits.  Then the
canonical functor $\bfK(\A)\to\bfD(\A)$ admits a fully faithful right
adjoint. In particular, the category $\bfD(\A)$ has small Hom-sets.
\end{thm}

The proof of this theorem is based on the following lemma.

\begin{lem}\label{le:derived}
Let $\A$ be an exact category that is locally $\a$-presentable for
some regular cardinal $\a$.  Suppose that exact sequences are closed under
$\a$-filtered colimits. Then the following holds.
\begin{enumerate}
\item Let $X'\xto{f'} Y'$ be an $\a$-pure subobject of $X\xto{f} Y$ in
  $\Mor\A$. Then $\Ker f'\to X$ is an $\a$-pure subobject of $\Ker
  f\to X$, and $Y\to \Coker f'$ is an $\a$-pure subobject of $Y\to
  \Coker f$.
\item The isomorphisms in $\A$ are closed under $\a$-pure
  subobjects in $\Mor\A$.
\item The admissible monomorphisms in $\A$ are closed under $\a$-pure
  subobjects in $\Mor\A$.
\item $\bfC_\ac(\A)$ is closed under $\a$-pure subobjects in
  $\bfC(\A)$.
\end{enumerate}
\end{lem}
\begin{proof}
(1) This follows from the fact that kernels and cokernels are
  preserved by taking $\a$-filtered colimits \cite[Korollar~7.12]{GU}.

(2) Let $f$ be an $\a$-pure subobject of an isomorphism. It follows
  from (1) that $f$ is an epimorphism. On the other hand, there is a
  morphism $g$ such that the composite $gf$ is an $\a$-pure
  monomorphism. Thus $gf$ is a regular monomorphism
  \cite[Proposition~2.31]{AdRo} and therefore extremal. It follows
  that $f$ is an isomorphism.

(3) Consider an admissible monomorphism $X\to Y$ and an $\a$-pure
  subobject $X'\to Y'$. Thus there is a commuting square
$$\xymatrix{X'\ar[r]\ar[d]&Y'\ar[d]\\ X\ar[r]&Y}
$$ such that the vertical morphisms are $\a$-pure monomorphisms. Any
split monomorphism is admissible (property a) of an exact category in
\cite[\S2]{Q}), and therefore every $\a$-pure monomorphism is
admissible, since exact sequences are closed under $\a$-filtered
colimits. It follows that the composite $X'\to Y'\to Y$ is an
admissible monomorphism (property b) in \cite[\S2]{Q}). Thus $X'\to
Y'$ is an admissible monomorphism (property c) in \cite[\S2]{Q}). 

(4) First observe that a complex $X=(X^n,d^n)$ is acyclic if and only if
for each $n\in\bbZ$ the morphism $\Coker d^{n-2}\to\Ker d^n$ is
invertible and the monomorphism $\Ker d^n\to X^n$ is admissible.

Now fix an $\a$-pure monomorphism $X\to\bar X$ in $\bfC(\A)$ such that
$\bar X$ is acyclic. It follows from (1) that $\Coker d^{n-2}\to\Ker
d^n$ is an $\a$-pure subobject of $\Coker \bar d^{n-2}\to\Ker \bar
d^n$ in $\Mor\A$, and that $\Ker d^n\to X^n$ is an $\a$-pure subobject
of $\Ker \bar d^n\to \bar X^n$. Isomorphisms and admissible
monomorphisms in $\A$ are closed under $\a$-pure subobjects by (2) and
(3). We conclude that $X$ is acyclic.
\end{proof}

\begin{proof}[Proof of Theorem~\ref{th:derived}]
Consider the full subcategory $\bfC_\ac(\A)$ of acyclic complexes in
$\bfC(\A)$.  The assumption on the exact structure implies that
$\bfC_\ac(\A)$ is closed under filtered colimits, and
Lemma~\ref{le:derived} implies that $\bfC_\ac(\A)$ is closed under
$\a$-pure subobjects for some regular cardinal $\a$.  Thus every
object in $\bfC(\A)$ admits a right $\bfC_\ac(\A)$-approximation by
Theorem~\ref{th:approx}. Applying Proposition~\ref{pr:N}, it follows
that the inclusion $\bfK_\ac(\A)\to\bfK(\A)$ admits a right adjoint.
A standard argument \cite[Proposition~9.1.18]{Ne} then shows that the
quotient functor $\bfK(\A)\to\bfD(\A)$ admits a right adjoint. Note
that this right adjoint is fully faithful
\cite[Proposition~I.1.3]{GZ}. Therefore $\bfD(\A)$ has small Hom-sets.
\end{proof}

The following corollary is a straightforward generalisation of
Theorem~\ref{th:derived}; its proof requires only minor modifications.

\begin{cor}\label{co:derived}
Let $\A$ be an exact category that is locally
presentable, and suppose that exact sequences are closed under
filtered colimits.  Let $\B$ be a full additive subcategory of $\A$
that is extension closed, closed under filtered colimits, and closed
under $\a$-pure subobjects or under $\a$-pure quotients for some
regular cardinal $\a$.  Then the canonical functor
$\bfK(\B)\to\bfD(\B)$ admits a fully faithful right adjoint.
\end{cor}

\begin{proof}
We follow the proof of Theorem~\ref{th:derived} and assume first that
$\B$ is closed under $\a$-pure subobjects. 

As before, one shows that $\bfC_\ac(\B)$ is closed under filtered
colimits and $\a$-pure subobjects, viewed as a subcategory of
$\bfC(\A)$.  Some extra care is needed for the fact that admissible
monomorphisms in $\B$ are closed under $\a$-pure subobjects in
$\Mor\A$. Note that a morphism in $\B$ is an admissible monomorphism
if and only if it is an admissible monomorphism in $\A$ and its
cokernel belongs to $\B$. Thus we consider an admissible monomorphism
$X\xto{f} Y$ in $\B$ and an $\a$-pure subobject $X'\xto{f'} Y'$. Then
$f'$ is an admissible monomorphism in $\A$ by Lemma~\ref{le:derived},
and it belongs to $\B$ since $\B$ is closed under $\a$-pure
subobjects. Moreover, $\Coker f'$ is an $\a$-pure subobject of $\Coker
f$ and belongs therefore to $\B$. It follows that $f'$ is an
admissible monomorphism in $\B$.

The rest of the proof goes as before. Thus every object in $\bfC(\B)$
admits a right  $\bfC_\ac(\B)$-approximation, and it follows that
the inclusion $\bfK_\ac(\B)\to\bfK(\B)$ admits a right adjoint.

The case that $\B$ is closed under $\a$-pure quotients is similar and
therefore left to the reader.
\end{proof}

\begin{rem}
It seems to be an interesting project to establish in the context of
Theorem~\ref{th:derived} a Quillen model structure on the category
$\bfC(\A)$ such that cofibrations are the degreewise admissible
monomorphisms and weak equivalences are the quasi-isomorphisms. This
would extend the work of Beke in \cite{Be}.  A strategy for this
programme has been pointed out by Maltsiniotis \cite{Ma}.
\end{rem}

The following example has been studied by Neeman \cite{Ne2008,Ne2006}
and Murfet \cite{Mu}.

\begin{exm}
Let $\bbX=(\bbX,\Oc_\bbX)$ be a scheme. The flat $\Oc_\bbX$-modules
form an extension closed subcategory of $\Qcoh \bbX$. Thus
Corollary~\ref{co:derived} can be applied. It follows that the
canonical functor $\bfK(\Flat \bbX)\to\bfD(\Flat \bbX)$ admits a right
adjoint.

Let $\Inj\bbX$ denote the full subcategory consisting of all injective
modules in $\Qcoh \bbX$.  If $\bbX$ is noetherian and separated, then
tensoring with a dualising complex induces an equivalence $\bfD(\Flat
\bbX)\xto{\sim}\bfD(\Inj\bbX)$; this is an `infinite completion' of
Grothendieck duality.
\end{exm}

\section{The pure derived category}

In this section we investigate the derived category of an additive
category with respect to its pure exact structure. It seems natural to
focus on this exact structure because it is the smallest one such that
exact sequences are closed under filtered colimits. An example of
particular interest is the pure derived category of a module category.

Locally finitely presented categories in the sense of Crawley-Boevey
form a convenient setting for studying purity \cite{CB}. Thus we fix
an additive category $\A$ that is \emph{locally finitely presented}.
This means $\A$ admits filtered colimits and every object in $\A$ can
be written as a filtered colimit of some fixed set of finitely
presented objects. Recall that an object $X$ is \emph{finitely
  presented} if the functor $\Hom_\A(X,-)$ preserves filtered
colimits. Denote by $\fp\A$ the full subcategory formed by all
finitely presented objects. We consider the pure exact structure, that
is, a sequence $X\to Y\to Z$ of morphisms in $\A$ is \emph{pure exact}
if the induced sequence
$0\to\Hom_\A(C,X)\to\Hom_\A(C,Y)\to\Hom_\A(C,Z)\to 0$ is exact for
each finitely presented object $C$. The projective objects with
respect to this exact structure are called \emph{pure projective};
they are precisely the direct summands of coproducts of finitely
presented objects. The derived category with respect to this exact
structure is by definition the \emph{pure derived category}.

The next result combines Corollary~\ref{co:derived} with recent work
of Neeman \cite{Ne2008}.

\begin{thm}\label{th:pure}
Let $\A$ be a locally finitely presented additive category, endowed
with the pure exact structure. Then there exists the following
recollement.
$$\xymatrixrowsep{3pc} \xymatrixcolsep{2pc}\xymatrix{
\bfK_\ac(\A)\,\ar[rr]|-{\inc}&&\,\bfK(\A)\, \ar[rr]|-{\can}
\ar@<1.2ex>[ll]^-{}\ar@<-1.2ex>[ll]_-{}&&
\,\bfD(\A)\ar@<1.2ex>[ll]^-{}\ar@<-1.2ex>[ll]_-{} }$$ Moreover, the
composite $\bfK(\Proj\A)\xto{\inc}\bfK(\A)\xto{\can} \bfD(\A)$
is an equivalence.
\end{thm}

\begin{proof}
We need to show that the functors $\bfK_\ac(\A)\to\bfK(\A)$ and
$\bfK(\A)\to\bfD(\A)$ admit left and right adjoints.

Set $\C=\fp\A$ and consider the category $\Mod\C$ of $\C$-modules,
that is, additive functors $\C^\op\to\Ab$. The functor $\A\to \Mod\C$
taking an object $X$ to $\Hom_\A(-,X)|_\C$ is fully faithful and
preserves filtered colimits; it identifies $\A$ with the full
subcategory $\Flat\C$ of flat $\C$-modules \cite[\S1.4]{CB}. It
follows from Corollary~\ref{co:derived} that the functors
$\bfK_\ac(\A)\to\bfK(\A)$ and $\bfK(\A)\to\bfD(\A)$ admit right
adjoints.

The other half of the recollement has been established in
\cite{Ne2008}. In fact, we can identify the category $\Proj\C$ of
projective $\C$-modules with $\Proj\A$. It follows from
\cite[Proposition~8.1]{Ne2008} that the inclusion
$\bfK(\Proj\A)\to\bfK(\A)$ admits a right adjoint. Moreover,
$\bfK(\Proj\A)^\perp=\bfK_\ac(\A)$ by \cite[Theorem~8.6]{Ne2008}, that
is, an object $Y$ in $\bfK(\A)$ is acyclic if and only if
$\Hom_{\bfK(\A)}(X,Y)=0$ for all $X$ in $\bfK(\Proj\A)$. Using
standard arguments \cite[\S9]{Ne}, it follows that the inclusion
$\bfK_\ac(\A)\to\bfK(\A)$ admits a left adjoint and that the composite
$\bfK(\Proj\A)\xto{\inc}\bfK(\A)\xto{\can} \bfD(\A)$ is an
equivalence.
\end{proof}

Let us reformulate this result in the form which is due to Neeman.

\begin{cor}[{\cite[Remark~3.2]{Ne2006}}] 
\pushQED{\qed} Let $A$ be a ring with several objects. Then there
exists the following recollement.
$$\xymatrixrowsep{3pc} \xymatrixcolsep{2pc}\xymatrix{ \bfK_\ac(\Flat
A)\,\ar[rr]|-{\inc}&&\,\bfK(\Flat A)\, \ar[rr]|-{}
\ar@<1.2ex>[ll]^-{}\ar@<-1.2ex>[ll]_-{}&& \,\bfK(\Proj
A)\ar@<1.2ex>[ll]^-{}\ar@<-1.2ex>[ll]|-{\inc} }$$ Moreover, the
composite $\bfK(\Proj A)\xto{\inc}\bfK(\Flat A)\xto{\can} \bfD(\Flat A)$ is an
equivalence.\qedhere
\end{cor}

It is a remarkable fact that we have an equivalence $\bfK(\Proj
A)\xto{\sim} \bfD(\Flat A)$, even though it may
happen that flat $A$-modules have infinite projective dimension.  Thus
it would be interesting to have necessary and sufficient conditions for
an exact category $\A$ having enough projective objects, such that the
composite $\bfK(\Proj\A)\xto{\inc}\bfK(\A)\xto{\can} \bfD(\A)$ is an
equivalence.

Let $A$ be a ring. We denote by $\mod A$ the category of finitely
presented $A$-modules and let $\proj A=\Proj A\cap\mod A$. Suppose
that $A^\op$ is \emph{coherent}, that is, the category $\mod A^\op$ is
abelian. Then we have the following description of the compact objects
of $\bfK(\Proj A)$ which is due to J{\o}rgensen and Neeman
\cite{Jo,Ne2008}. 

Given any triangulated category $\T$, we denote by
$\T^c$ the full subcategory formed by all compact objects.
 
\begin{prop}[{\cite[Proposition~7.14]{Ne2008}}]\label{pr:comp}
\pushQED{\qed} Let $A$ be a ring with several objects and suppose that
$A^\op$ is coherent. Then the triangulated category $\bfK(\Proj A)$ is
compactly generated and there is an equivalence
\[\bfD^b(\mod A^\op)^\op\xto{\sim}\bfK^{-,b}(\proj
A^\op)^\op\xto{\Hom_{A^\op}(-,A)}\bfK(\Proj A)^c.\qedhere\]
\end{prop}

Let $\A$ be a locally finitely presented additive category. In
\cite{CB}, Crawley-Boevey showed that $\A$ admits set-indexed products
iff $\fp\A$ admits pseudo-cokernels iff the category
$\widecheck{\fp\A}$ is abelian, where $\widecheck\C=(\mod\C^\op)^\op$
for any additive category $\C$.

\begin{thm}\label{th:pure-der}
Let $\A$ be a locally finitely presented additive category, endowed
with the pure exact structure. Suppose that $\A$ admits set-indexed
products.  Then the derived category $\bfD(\A)$ is a compactly
generated triangulated category and the inclusion $\fp\A\to\A$ induces
an equivalence
$$\bfD^b(\widecheck{\fp\A})\xto{\sim}\bfD(\A)^c.$$
\end{thm}
\begin{proof} 
View $\C=\fp\A$ as a ring with several objects and identify the
category $\Proj\C$ of projective $\C$-modules with $\Proj\A$; see the
proof of Theorem~\ref{th:pure}. Now combine the equivalence
$\bfK(\Proj\A)\xto{\sim}\bfD(\A)$ from Theorem~\ref{th:pure} with the
description of the compact objects of $\bfK(\Proj\C)$ given in
Proposition~\ref{pr:comp}.
\end{proof}

\begin{rem}\label{re:pure-der}
Let $\A$ be a locally finitely presented additive category that is
cocomplete. Then the category $\C=\fp\A$ admits cokernels, and it
follows that each object in $\widecheck{\C}$ has injective dimension
at most two.  Thus the inclusion $\C\to\widecheck{\C}$ induces an
equivalence $\bfK^b(\C)\xto{\sim}\bfD^b(\widecheck{\C})$, and
therefore the inclusion $\fp\A\to\A$ induces an equivalence
$\bfK^b(\fp\A)\xto{\sim}\bfD(\A)^c$.
\end{rem}

Let us exhibit the pure derived category of a module category more
closely; see also \cite{ChHo,Ro2002}. Fix a ring $A$. The category
$\A=\Mod A$ of $A$-modules is locally finitely presented and
$\fp\A=\mod A$. Note that $\widecheck{\mod A}$ equals the \emph{free
abelian category} $\Ab(A)$ over $A$, where the ring $A$ is viewed as a
category with a single object \cite{Gr}. To be precise, the functor
$A\to\Ab(A)$ taking $A$ to $\Hom_A(A,-)$ has the property that any
additive functor $A\to\C$ to an abelian category extends uniquely, up
to a unique isomorphism, to an exact functor $\Ab(A)\to\C$.

We denote by $\bfD_\pure(\Mod A)$ the \emph{pure derived category} of
$\Mod A$, that is, the derived category with respect to the pure exact
structure on $\Mod A$. The usual derived category with respect to all
exact sequences in $\Mod A$ is denoted by $\bfD(\Mod A)$.  Recall that
$\bfD(\Mod A)$ is a compactly generated triangulated category with an
equivalence $\bfK^b(\proj A)\xto{\sim}\bfD(\Mod A)^c$ \cite{Kel1994}.
We have the following analogue for the pure derived category.

\begin{cor}\label{co:compgen}
Let $A$ be a ring.  The pure derived category $\bfD_\pure(\Mod A)$ is
a compactly generated triangulated category and the inclusions $\mod
A\to\Mod A$ and $\mod A\to\Ab (A)$ induce equivalences
$$\bfD^b(\Abel(A))\xleftarrow{\sim}\bfK^b(\mod
A)\xto{\sim}\bfD_\pure(\Mod A)^c.$$ The canonical functor
$\bfD_\pure(\Mod A)\to\bfD(\Mod A)$ admits left and right adjoints
that are fully faithful.  The left adjoint preserves compactness and
its restriction to compact objects identifies with the inclusion
$\bfK^b(\proj A)\to\bfK^b(\mod A)$.
\end{cor}
\begin{proof}
Let us write $\A=\Mod A$.  The fact that $\bfD_\pure(\A)$ is compactly
generated and the description of the compact objects follow from
Theorem~\ref{th:pure-der} and Remark~\ref{re:pure-der}.  The canonical
functor
$$F\colon
\bfD_\pure(\A)=\bfK(\A)/\bfK_\pac(\A)\lto\bfK(\A)/\bfK_\ac(\A)=\bfD(\A)$$
preserves set-indexed (co)products and admits therefore a left adjoint
and a right adjoint, by Brown representability. These adjoints are
fully faithful since $F$ is a quotient functor
\cite[Proposition~I.1.3]{GZ}. The left adjoint preserves compactness
since $F$ preserves set-indexed coproducts.  Using
Theorem~\ref{th:pure}, we may identify in $\bfK(\A)$
$$\bfD_\pure(\A)={^\perp\bfK_\pac(\A)}\quad\text{and}\quad\bfD(\A)={^\perp\bfK_\ac(\A)}.$$
With this identification, the left adjoint of $F$ embeds $\bfK^b(\proj
A)$ into $\bfK^b(\mod A)$.
\end{proof}

\begin{rem}
The derived categories $\bfD(\Mod A)$ and $\bfD_\pure(\Mod A)$ are two
extremes. More precisely, the exact structures on $\Mod A$ are
partially ordered by inclusion. The natural exact structure given by
all kernel-cokernel pairs is the unique maximal one, while the pure
exact structure is the smallest exact structure such that exact
sequences are closed under filtered colimits.
\end{rem}

Given any ring $A$, we  view the category $\mod A$ as a
ring with several objects and denote it by $\widehat A$.

\begin{cor}\label{co:pder}
Let $A$ be a ring.  The fully faithful functor $\Mod A\to \Mod\widehat
A$ sending $X$ to $\Hom_A(-,X)|_{\widehat A}$ induces an equivalence
$\bfD_\pure(\Mod A)\xto{\sim} \bfD(\Mod\widehat A)$.
\end{cor}
\begin{proof}
The functor $\Mod A\to \Mod\widehat A$ sends pure exact sequences to
exact sequences and induces therefore an exact functor $F\colon
\bfD_\pure(\Mod A)\to \bfD(\Mod\widehat A)$. The description of the
compact objects in Corollary~\ref{co:compgen} implies that $F$
restricts to an equivalence between the subcategories of compact
objects; thus $F$ is an equivalence by a standard devisage argument.
\end{proof}

It seems interesting to find out when two rings $A$ and $B$ have
equivalent pure derived categories. In view of
Corollary~\ref{co:pder}, this reduces to the question when $\widehat
A$ and $\widehat B$ have equivalent derived categories; thus tilting
theory applies \cite{Ri,Ke}.

\begin{exm}
Fix a field $k$ and consider the path algebras $A=k\varGamma$ and
$B=k\varDelta$ of the following quivers.
\[\xymatrix @!0 @R=1.8em @C=2.6em {
 &\circ\ar[rd]&&&&&\circ\ar[rd]\\
 \varGamma\colon&&\circ\ar[r]&\circ &&\varDelta\colon&&\circ&\ar[l]\circ\\
&\circ\ar[ru]&&&&&\circ\ar[ru]}\]
Note that both algebras are of finite representation type. Thus
$\widehat A$ and $\widehat B$ are each Morita equivalent to their
associated Auslander algebra; see \cite{Be1992} for unexplained terminology.

It follows from work of Ladkani \cite{La2009} that $\widehat A$ and
$\widehat B$ are both derived equivalent to the incidence algebra $kP$ of
the poset $P=D_4\times A_3$.  Thus we have equivalences
$$\bfD_\pure(\Mod A)\xto{\sim}\bfD(\Mod
kP)\xleftarrow{\sim}\bfD_\pure(\Mod B)$$ even though the categories
$\Mod A$ and $\Mod B$ are not equivalent.
\end{exm}

\section{Left approximations and left adjoints}

In this section we discuss briefly the existence of left
approximations and left adjoints. In fact, most results are parallel
to those previously obtained for right approximations and right
adjoints. Given any category $\A$ and a full subcategory $\B$, a
morphism $f\colon X\to Y$ in $\A$ is called \emph{left
$\B$-approximation} of $X$ if $Y$ belongs to $\B$ and every morphism
$X\to Y'$ with $Y'$ in $\B$ factors through $f$.

The following is the principal existence result for left
approximations; it is the analogue of Theorem~\ref{th:approx}. The
result is well-known for subcategories of a module category that are
closed under $\aleph_0$-pure submodules \cite{Ki,RS}.

\begin{prop}\label{pr:leftapprox}
Let $\A$ be a locally presentable category and $\B$ be a full
subcategory. Suppose that $\B$ is closed under set-indexed products
and $\a$-pure subobjects for some regular cardinal $\a$.  Then each
object in $\A$ admits a left $\B$-approximation.
\end{prop}
\begin{proof}
Fix an object $X$ in $\A$.  We may assume that $\A$ is locally
$\a$-presentable and that $X$ is $\b$-presentable for some regular
cardinal $\b\ge\a$. In \cite[Theorem~2.33]{AdRo} it is shown that each
morphism $X\to Y$ factors through an $\a$-pure monomorphism $Y'\to Y$
such that $Y'$ is $\b$-presentable.  Choose a representative set of
morphisms $f_i\colon X\to Y_i$ ($i\in I$) with $Y_i$ in $\B$ and
$\b$-presentable.  Then it follows from the assumption on $\B$ that
the induced morphism $X\to\prod_{i\in I} Y_i$ is a left
$\B$-approximation.
\end{proof}

The next lemma will be used in some of the following applications.

\begin{lem}\label{le:ext}
Let $\A$ be a locally presentable abelian category and $Z$ an object
in $\A$. Then there exists a regular cardinal $\a$ such that the
kernel of $\Ext^1_\A(Z,-)$ is closed under $\a$-pure subobjects.
\end{lem}
\begin{proof}
Choose a regular cardinal $\b$ such that $\A$ is locally
$\b$-presentable and $Z$ is $\b$-presentable. Then there exists a
regular cardinal $\a\ge\b$ such that the kernel of each morphism $Y\to
Z$ from a $\b$-presentable object $Y$ is $\a$-presentable.  

Let $X\to \bar X$ be an $\a$-pure monomorphism such that
$\Ext_\A^1(Z,\bar X)=0$ and fix an exact sequence $\eta\colon 0\to
X\to Y\to Z\to 0$ in $\A$. The choice of $\a$ implies that the sequence
$\eta$ fits into a commutative diagram of the following form
$$\xymatrix{ 0\ar[r]&X'
\ar[r]\ar[d]&Y'\ar[r]\ar[d]&Z\ar[r]\ar@{=}[d]&0\\
0\ar[r]&X\ar[r]&Y\ar[r]&Z\ar[r]&0}$$ such that the upper row is exact
and consists of $\a$-presentable objects. The composite $X'\to X\to
\bar X$ factors through $X'\to Y'$, since $\Ext_\A^1(Z,\bar X)=0$. It
follows that $X'\to X$ factors through $X'\to Y'$, since $X\to \bar X$
is an $\a$-pure monomorphism. Thus the sequence $\eta$ splits, and we conclude that
$\Ext_\A^1(Z,X)=0$.
\end{proof}

The following example is our first application of
Proposition~\ref{pr:leftapprox}.

\begin{exm}
Let $\A$ be a locally presentable abelian category and $\C$ a set of objects.
Then the full subcategory
$$\C'=\{X\in\A\mid\Ext_\A^1(Z,X)=0\text{ for all }Z\in\C\}$$ is closed
under set-indexed products and $\a$-pure subobjects for some regular
cardinal $\a$, by Lemma~\ref{le:ext}. It follows from
Proposition~\ref{pr:leftapprox} that each object in $\A$ admits a left
$\C'$-approximation.
\end{exm}

As before, we extend the existence of approximations to homotopy categories.

\begin{cor}\label{co:leftadj}
Let $\A$ be a locally presentable additive category and $\B$ a full
additive subcategory. Suppose that $\B$ is closed under set-indexed
products and $\a$-pure subobjects for some regular cardinal
$\a$.  Then the inclusion $\bfK(\B)\to\bfK(\A)$ admits a left adjoint.
\end{cor}
\begin{proof}
We view $\bfC(\B)$ as a full subcategory of $\bfC(\A)$. (Co)limits in
$\bfC(\A)$ are computed degreewise and this implies that $\bfC(\B)$ is
closed under set-indexed products and $\a$-pure subobjects.  Thus
every object in $\bfC(\A)$ admits a left $\bfC(\B)$-approximation by
Proposition~\ref{pr:leftapprox}. Applying the dual statement of
Proposition~\ref{pr:N}, it follows that the inclusion
$\bfK(\B)\to\bfK(\A)$ admits a left adjoint.
\end{proof}

\begin{exm}
Let $\A$ be a Grothendieck abelian category. Then the full subcategory
$\Inj\A$ consisting of all injective objects is closed under $\a$-pure
subobject for some regular cardinal $\a$. This follows from
Lemma~\ref{le:ext} and a variant of Baer's criterion, because there is
a set of objects $\C$ such that an object $X$ in $\A$ is injective if
and only if $\Ext^1_\A(Z,X)=0$ for all $Z$ in $\C$. It follows from
Corollary~\ref{co:leftadj} that the inclusion
$\bfK(\Inj\A)\to\bfK(\A)$ admits a left adjoint.
\end{exm}

The following result about derived categories is the analogue of
Corollary~\ref{co:derived}; the proof is almost the same and therefore
left to the reader.

\begin{cor}\label{co:leftderived}
\pushQED{\qed} Let $\A$ be an exact category that is locally
presentable and let $\a$ be a regular cardinal.  Suppose that exact
sequences are closed under set-indexed products and $\a$-filtered
colimits. Let $\B$ be a full subcategory of $\A$ that is closed under
extensions, set-indexed products, and $\a$-pure subobjects. Then the
canonical functor $\bfK(\B)\to\bfD(\B)$ admits a fully faithful left
adjoint. In particular, the category $\bfD(\B)$ has small
Hom-sets.\qedhere
\end{cor}

\begin{exm}
Let $\A$ be a locally $\a$-presentable additive category. Consider the
$\a$-pure exact structure, that is, a sequence $X\to Y\to Z$ of
morphisms in $\A$ is \emph{$\a$-pure exact} if the induced sequence
$0\to\Hom_\A(C,X)\to\Hom_\A(C,Y)\to\Hom_\A(C,Z)\to 0$ is exact for
each $\a$-presentable object $C$. The $\a$-pure exact sequences are
closed under set-indexed products and $\a$-filtered colimits. In fact,
a sequence is $\a$-pure exact if and only if it is an $\a$-filtered
colimit of split exact sequences. It follows from
Corollary~\ref{co:leftderived} that the canonical functor
$\bfK(\A)\to\bfD(\A)$ admits a left adjoint.
\end{exm}

\begin{exm}
Let $A$ be a ring. Consider the category $\Fpinj A$ of all
\emph{fp-injective} $A$-modules, that is, $A$-modules $X$ such that
$\Ext^1_A(-,X)$ vanishes on all finitely presented $A$-modules.  Note
that fp-injective modules are precisely the pure submodules of
injective modules, whereas flat modules are the pure quotients of
projective modules.  The fp-injective $A$-modules form an extension
closed subcategory of $\Mod A$ that is closed under set-indexed
products and pure submodules.  It follows from
Corollary~\ref{co:leftderived} that the canonical functor $\bfK(\Fpinj
A)\to\bfD(\Fpinj A)$ admits a left adjoint.
\end{exm}

Using fp-injective modules, a result from \cite{IK} takes the
following form. It seems appropriate to mention this because it
stimulated our interest in adjoint functors between homotopy
categories.
\medskip

\noindent\emph{Given any pair $A,B$ of noetherian rings that admit a
dualising complex $D$, there are equivalences
\[\xymatrix{
\bfK(\Proj A)\ar[r]^-\sim&\bfD(\Flat A)\ar@<.7ex>[rr]^-{-\otimes_A
D}&&\bfD(\Fpinj B)\ar@<.7ex>[ll]^-{\Hom_B(D,-)}&\bfK(\Inj
B)\ar[l]_-\sim }.\]}

Let us conclude this note with the following conjecture; it is in
analogue of results on flat modules in \cite{Ne2008}.
\begin{conj}\label{conj:coh}
Given any ring $A$, the composite $\bfK(\Inj
A)\xto{\inc}\bfK(\Fpinj A)\xto{\can}\bfD(\Fpinj A)$ is an
equivalence. If $A$ is coherent, then $\bfD(\Fpinj A)$ is compactly
generated and the composite $\bfD(\Fpinj A)\to\bfD_\pure(\Mod
A)\to\bfD(\Mod A)$ induces an equivalence $\bfD(\Fpinj
A)^c\xto{\sim}\bfD^b(\mod A)$.
\end{conj}

This conjecture should be formulated more generally as follows. Let
$\C$ be a skeletally small abelian category and let
$\A=\Ex(\C^\op,\Ab)$ denote the category of exact functors
$\C^\op\to\Ab$. Categories of this form are ubiquitous
\cite{Kr1998,Pr2009}.  Note that $\A$ is the category of fp-injective
objects of the locally coherent Grothendieck category
$\Lex(\C^\op,\Ab)$ of left exact functors $\C^\op\to\Ab$.  The
category $\A$ admits set-indexed products and filtered
colimits. Moreover, $\A$ admits a canonical exact structure with
enough injective objects. To be precise, a sequence $X\to Y\to Z$ of
morphisms in $\A$ is \emph{exact} if the induced sequence $0\to ZC\to
YC\to XC\to 0$ of abelian groups is exact for all $C$ in $\C$.

The general form of the above conjecture then says that the canonical
functor $\bfK(\Inj\A)\to \bfD(\A)$ is an equivalence and that
$\bfD(\A)$ is compactly generated with an equivalence
$\bfD(\A)^c\xto{\sim}\bfD^b(\C)$. This conjecture specialises to
Conjecture~\ref{conj:coh} by taking for $\C$ the free abelian category
$\Ab(A)$ over a ring $A$; it has been proved in \cite{Kr2005} provided
that each object in $\C$ is noetherian; see Theorem~\ref{th:pure-der}
for another special case.

\end{document}